\documentclass[12pt]{amsart}

\usepackage{amssymb}
\usepackage{fullpage}

\usepackage{xcolor}
\usepackage{hyperref}
\hypersetup{colorlinks=true,citecolor=purple,linkcolor=purple,urlcolor=gray,breaklinks=true}

\numberwithin{equation}{section}
\newtheorem{thm}{Theorem}[section]
\newtheorem{prop}[thm]{Proposition}
\newtheorem{lem}[thm]{Lemma}
\newtheorem{cor}[thm]{Corollary}

\newtheorem{prob}[thm]{Problem}
\theoremstyle{definition}
\newtheorem{definition}[thm]{Definition}
\newtheorem{remark}[thm]{Remark}
\newtheorem{example}[thm]{Example}

\def\ldiv{\backslash}
\def\rdiv{/}

\newcommand{\vhi}{\varphi}
\newcommand{\m}{^{-1}}

\newcommand{\com}{\operatorname{Com}}
\newcommand{\sym}{\operatorname{Sym}}
\newcommand{\mlt}{\operatorname{Mlt}}
\newcommand{\inn}{\operatorname{Inn}}
\newcommand{\atp}{\operatorname{Atp}}

\newcommand{\aut}{\operatorname{Aut}}
\newcommand{\id}{\operatorname{id}}

\newcommand{\nuc}{\operatorname{Nuc}}
\newcommand{\lps}{\operatorname{Psa}_\ell}

\begin{document}

\title{Abelian congruences and solvability\\ in Moufang loops}

\author{Ale\v{s} Dr\'{a}pal}
\address[Dr\'apal]{Department of Algebra\\ Faculty of Mathematics and Physics\\ Charles University\\ Sokolovsk\'a 83\\ 186 75 Praha 8, Czech Republic}
\email[Dr\'apal]{drapal@karlin.mff.cuni.cz}

\author{Petr Vojt\v echovsk\'y}
\address[Vojt\v{e}chovsk\'y]{Department of Mathematics\\ University of Denver\\ 2390 S.~York St.\\ Denver, CO 80208, USA}
\email[Vojt\v{e}chovsk\'y]{petr@math.du.edu}

\thanks{A.~Dr\'apal supported by the INTER-EXCELLENCE project LTAUSA19070 of M\v SMT Czech Republic. P.~Vojt\v{e}chovsk\'y supported by the Simons Foundation Mathematics and Physical Sciences Collaboration Grant for Mathematicians no.~855097 and by the PROF grant of the University of Denver.}

\begin{abstract}
In groups, an abelian normal subgroup induces an abelian congruence. We construct a class of centrally nilpotent Moufang loops containing an abelian normal subloop that does not induce an abelian congruence. On the other hand, we prove that in $6$-divisible Moufang loops, every abelian normal subloop induces an abelian congruence.

In loops, congruence solvability adopted from the universal-algebraic commutator theory of congruence modular varieties is strictly stronger than classical solvability adopted from group theory. It is an open problem whether the two notions of solvability coincide in Moufang loops. We prove that they coincide in $6$-divisible Moufang loops and in Moufang loops of odd order. In fact, we show that every Moufang loop of odd order is congruence solvable, thus strengthening Glauberman's Odd Order Theorem for Moufang loops.
\end{abstract}



\maketitle

We investigate abelian normal subloops and the theory of solvability in Moufang loops. There are two notions of solvability in loop theory, one adopted from solvability in group theory, called \emph{classical solvability} here, and another adopted from commutator theory in congruence modular varieties, called \emph{congruence solvability} here.

When translated into the language of normal series $Q=Q_0\ge Q_1\ge \cdots \ge Q_n=1$, the difference between the two solvability notions is that classical solvability requires all factors $Q_i/Q_{i+1}$ to be abelian (that is, commutative groups), while congruence solvability requires a potentially stronger condition, namely that every factor $Q_i/Q_{i+1}$ induces an abelian congruence of $Q/Q_{i+1}$. Whether a normal subloop of a loop $Q$ is merely abelian or whether it induces an abelian congruence of $Q$ can be seen on the level of multiplication tables, which must have a more rigid structure in the latter case, cf.~Subsection \ref{Ss:Abelianess}.

Every congruence solvable loop is classically solvable but the converse is not true. The two notions of solvability coincide in groups. More generally, they coincide in any loop $Q$ in which every abelian normal subloop induces an abelian congruence of $Q$. The situation is delicate, however, since it is certainly possible for a loop $Q$ to be congruence solvable, yet posses an abelian normal subloop that does not induce an abelian congruence of $Q$.

It is an open problem whether the two notions of solvability coincide in Moufang loops. In this paper we offer a general construction of centrally nilpotent Moufang loops that contain an abelian normal subloop that does not induce an abelian congruence, cf.~Proposition \ref{Pr:Example}. On the other hand, we show that if $Q$ is a $3$-divisible Moufang loop and $X$ is a $2$-divisible abelian normal subloop of $Q$, then $X$ induces an abelian congruence of $Q$, cf.~Theorem \ref{Th:AbelianCongruence}. In particular, in a $6$-divisible Moufang loop, every abelian normal subloop induces an abelian congruence, cf.~Corollary \ref{Cr:AbelianCongruence}, and hence the two notions of solvability coincide in $6$-divisible Moufang loops, cf.~Corollary \ref{Cr:Coincide}. Up to that point, the exposition is self-contained (modulo basic results from loop theory) and the arguments are elementary in nature. This is in contrast with other results in Moufang loops, such as the Lagrange Theorem, whose only known proofs rely on the classification of finite simple groups.

Continuing, we build upon recent deep results of Cs\"org\H{o} \cite{Csorgo} on the nucleus in Moufang loops and prove that every Moufang loop of odd order is congruence solvable, cf.~Theorem \ref{Th:MoufOdd}. This strengthens a well-known result of Glauberman \cite{GlaubermanII} that states that every Moufang loop of odd order is classically solvable. Theorem \ref{Th:MoufOdd} implies the finitary version of Corollary \ref{Cr:Coincide}.

Background material on loops, Moufang loops and divisibility in power associative loops is collected in Section \ref{Sc:Background}. Abelianess and solvability for loops are discussed in Section \ref{Sc:Solvability}.

\section{Background on loops and Moufang loops}\label{Sc:Background}

\subsection{Loops}

See \cite{BruckBook} or \cite{PflugfelderBook} for an introduction to the theory of loops. A loop $Q$ is a magma $(Q,\cdot,1)$ with identity element $1$ such that all left translations $L_x:Q\to Q$, $L_x(y)=x\cdot y$ and all right translations $R_x:Q\to Q$, $R_x(y)=y\cdot x$ are permutations of $Q$.

We mostly denote the multiplication operation $\cdot$ by juxtaposition and we take advantage of $\cdot$ to indicate priority of multiplications in products, e.g., $x\cdot yz$ stands for $x\cdot(y\cdot z)$. The implicit division operations will be denoted by $x\ldiv y = L_x^{-1}(y)$ and $y\rdiv x = R_x^{-1}(y)$.

A mapping $f:Q_1\to Q_2$ of loops is a \emph{homomorphism} if $f(xy)=f(x)f(y)$ for all $x,y\in Q_1$. Then the identities $f(x\ldiv y) = f(x)\ldiv f(y)$ and $f(x\rdiv y) = f(x)\rdiv f(y)$ are automatically satisfied.

Let $\sym(Q)$ denote the symmetric group on $Q$. The \emph{multiplication group} of $Q$ is the subgroup
\begin{displaymath}
	\mlt(Q)=\langle L_x,R_x:x\in Q\rangle
\end{displaymath}
 of $\sym(Q)$. The \emph{inner mapping group} $\inn(Q)$ of $Q$ is the stabilizer of $1$ in $\mlt(Q)$. It is well known that
\begin{displaymath}
	\inn(Q)=\langle T_x,R_{x,y},L_{x,y}:x,y\in Q\rangle,
\end{displaymath}
where
\begin{displaymath}
    T_x = R_x^{-1}L_x,\quad R_{x,y}=R_{xy}^{-1}R_yR_x,\quad\text{and}\quad L_{x,y}=L_{xy}^{-1}L_xL_y.
\end{displaymath}
In groups, the inner mapping group is the familiar inner automorphism group. However, inner mappings of loops need not be automorphisms.

A subloop $X$ of a loop $Q$ is \emph{normal}, denoted by $X\unlhd Q$, if it is a kernel of a loop homomorphism. It turns out that a subloop $X\le Q$ is normal if and only if $f(X)=X$ for every $f\in\inn(Q)$.

The \emph{left nucleus}, \emph{middle nucleus} and the \emph{right nucleus} of $Q$ are the respective subloops
\begin{align*}
	\nuc_\ell(Q) &= \{x\in Q:x(yz)=(xy)z\text{ for all }y,z\in Q\},\\
	\nuc_m(Q) &=  \{x\in Q:y(xz)=(yx)z\text{ for all }y,z\in Q\},\\
	\nuc_r(Q) &=  \{x\in Q:y(zx)=(yz)x\text{ for all }y,z\in Q\}.
\end{align*}
The \emph{nucleus} $\nuc(Q)$ of $Q$ is the intersection of the above three nuclei. A subloop $X\le Q$ is \emph{nuclear} if $X\le\nuc(Q)$.

\begin{lem}\label{Lm:NucAutT}
Let $X$ be a normal subloop of a loop $Q$ such that $X\le\nuc_m(Q)\cap\nuc_r(Q)$. Then for every $u\in Q$ the inner mapping $T_u$ restricts to an automorphism of $X$
\end{lem}
\begin{proof}
Since $X\unlhd Q$, every inner mapping restricts to a permutation of $X$, and we only need to show that $T_u(xy) = T_u(x)T_u(y)$ holds for all $x,y\in X$. Hence we need to verify $(u\cdot xy)\rdiv u = ((ux)\rdiv u)((uy)\rdiv u)$ for every $x,y\in X$. This is equivalent to $u\cdot xy =  ((ux)\rdiv u)((uy)\rdiv u)\cdot u$. Since $y$ and $(uy)/u$ are elements of $X\le\nuc_m(Q)\cap\nuc_r(Q)$, we calculate
\begin{align*}
	 ((ux)\rdiv u)((uy)\rdiv u)\cdot u &= ((ux)\rdiv u)\cdot ((uy)\rdiv u)u = ((ux)\rdiv u)\cdot uy\\
		& = ((ux)\rdiv u)u\cdot y = (ux)y = u\cdot xy.\qedhere
\end{align*}
\end{proof}

Note that Lemma \ref{Lm:NucAutT} also holds under the dual assumption $X\unlhd Q$ and $X\le\nuc_\ell(Q)\cap\nuc_m(Q)$. See \cite[Lemma 1.7]{Drapal} for a slightly stronger statement.

\begin{cor}\label{Cr:NucAut}
Let $X$ be a nuclear normal subloop of a loop $Q$. Then every inner mapping of $Q$ restricts to an automorphism of $X$.
\end{cor}
\begin{proof}
Since $X\le\nuc_\ell(Q)\cap\nuc_r(Q)$, for every $u,v\in Q$ the inner mappings $L_{u,v}$, $R_{u,v}$ restrict to the identity mapping on $X$. We are done by Lemma \ref{Lm:NucAutT}.
\end{proof}

The \emph{center} of $Q$ is the subloop
\begin{displaymath}
	Z(Q)=\{x\in\nuc(Q):xy=yx\text{ for all }y\in Q\}.
\end{displaymath}
A \emph{central} subloop of $Q$ is a subloop of $Z(Q)$. Every central subloop of $Q$ is normal in $Q$.

A loop $Q$ is \emph{power associative} if every element of $Q$ generates an associative subloop of $Q$, that is, a subgroup. In particular, the powers $x^i$ of elements are well-defined in power associative loops, $x^{-1}=x\ldiv 1 = 1\rdiv x$, etc. A loop $Q$ is \emph{diassociative} if any two elements of $Q$ generate a subgroup.

A permutation $f$ of a loop $Q$ is a \emph{(left) pseudoautomorphism} of $Q$ if there exists $c\in Q$ such that
\begin{displaymath}
    cf(x)\cdot f(y) = cf(xy)
\end{displaymath}
for every $x,y\in Q$. The element $c$ is then called a \emph{(left) companion} of $f$. The set of all pairs $(c,f)\in Q\times\sym(Q)$, where $f$ is a pseudoautomorphism of $Q$ and $c$ is a companion of $f$, forms a group $\lps(Q)$ under the operations
\begin{equation}\label{Eq:LPs}
    (c,f)(d,g) = (cf(d),fg)\quad\text{and}\quad (c,f)^{-1} = (f^{-1}(c\ldiv 1),f^{-1}).
\end{equation}

A permutation $f\in\sym(Q)$ is a \emph{semiautomorphism} of $Q$ if $f(1)=1$ and
\begin{displaymath}
    f(x\cdot yx)= f(x)\cdot f(y)f(x)
\end{displaymath}
holds for all $x,y\in Q$. If $f$ is a semiautomorphism of a power associative loop $Q$ then an inductive argument shows that $f(x^i) = f(x)^i$ for every $i\in\mathbb Z$.

A triple $(f,g,h)$ of permutations of $Q$ is an \emph{autotopism} of $Q$ if $f(x)g(y) = h(xy)$ holds for all $x,y\in Q$. The autotopisms of $Q$ form a group $\atp(Q)$ under componentwise composition, the \emph{autotopism group} of $Q$. The following well-known result describes all autotopisms with a trivial component.

\begin{lem}\label{Lm:PrincipalAtp}
Let $Q$ be a loop. Then:
\begin{enumerate}
\item[(i)] $(\id_Q,g,h)\in\atp(Q)$ iff $g=h=R_x$ for some $x\in\mathrm{Nuc}_r(Q)$.
\item[(ii)] $(f,\id_Q,h)\in\atp(Q)$ iff $f=h=L_x$ for some $x\in\mathrm{Nuc}_\ell(Q)$.
\item[(iii)] $(f,g,\id_Q)\in\atp(Q)$ iff $f=R_x^{-1}$ and $g=L_x$ for some $x\in\mathrm{Nuc}_m(Q)$.
\end{enumerate}
\end{lem}

\begin{cor}\label{Cr:TrivNucAtp}
Let $Q$ be a loop. Then:
\begin{enumerate}
\item[(i)] If $\mathrm{Nuc}_r(Q)=1$ and $(f,g,h)\in\atp(Q)$ then $g$ and $h$ are determined by $f$.
\item[(ii)] If $\mathrm{Nuc}_\ell(Q)=1$ and $(f,g,h)\in\atp(Q)$ then $f$ and $h$ are determined by $g$.
\item[(iii)] If $\mathrm{Nuc}_m(Q)=1$ and $(f,g,h)\in\atp(Q)$ then $f$ and $g$ are determined by $h$.
\end{enumerate}
\end{cor}
\begin{proof}
Let us prove (i), the other parts being similar. Suppose that $\mathrm{Nuc}_r(Q)=1$ and $(f,g,h)$, $(f,\overline{g},\overline{h})\in\atp(Q)$. Then $(\id_Q,g^{-1}\overline{g},h^{-1}\overline{h}) = (f,g,h)^{-1}(f,\overline{g},\overline{h})\in\atp(Q)$. By Lemma \ref{Lm:PrincipalAtp}, $g^{-1}\overline{g} = h^{-1}\overline{h} = R_x$ for some $x\in\mathrm{Nuc}_r(Q)$. Since $\mathrm{Nuc}_r(Q) =1$, we have $R_x=R_1=\id_Q$, and $g=\overline{g}$, $h=\overline{h}$ follow.
\end{proof}

\subsection{Divisibility in power associative loops}\label{Ss:dDivisibility}

For an integer $d>1$ and a power associative loop $Q$, consider the mapping
\begin{equation}\label{Eq:dPower}
    h_d:Q\to Q,\quad x\mapsto x^d.
\end{equation}
In general, injectivity and surjectivity of $h_d$ are unrelated properties already in groups. (In the group of nonzero complex numbers under multiplication, $h_d$ is surjective but not injective. In the additive group of integers, $h_d$ is injective but not surjective.)

A power associative loop $Q$ is \emph{$d$-divisible} (resp. \emph{uniquely $d$-divisible}) if the mapping $h_d$ of \eqref{Eq:dPower} is surjective (resp. bijective).

\begin{lem}\label{Lm:dPowers}
Let $Q$ be a finite power associative loop, $d>1$ an integer and $h_d$ as in \eqref{Eq:dPower}. The following conditions are equivalent:
\begin{enumerate}
\item[(i)] $h_d$ is surjective on $Q$,
\item[(ii)] $h_d$ is injective on $Q$,
\item[(iii)] $Q$ contains no nonidentity element of order dividing $d$,
\item[(iv)] $Q$ contains no element of prime order dividing $d$.
\end{enumerate}
\end{lem}
\begin{proof}
Thanks to finiteness, (i) and (ii) are equivalent. If $Q$ contains an element $x\ne 1$ of order dividing $d$, then $h_d(x)=x^d = 1 = h_d(1)$ and $h_d$ is not injective. Hence (ii) implies (iii). Clearly, (iii) implies (iv). In fact, (iii) and (iv) are equivalent since if $1\ne x\in Q$ is such that $|x|$ divides $d$ and $p$ is a prime dividing $|x|$, then the cyclic group $\langle x\rangle$ contains an element of order $p$ (dividing $d$). Finally, suppose that (iii) holds, let $x\in Q$ and consider the cyclic group $C=\langle x\rangle$. Let $k=\mathrm{gcd}(|C|,d)$. If $k>1$ then $C$ contains a nonidentity element of order $k$ dividing $d$, a contradiction. Thus $\mathrm{gcd}(|C|,d)=1$ and $h_d$ restricts to a permutation of $C$. In particular, there is $y\in C$ such that $h_d(y)=x$, so $h_d$ is surjective on $Q$.\end{proof}

Given a prime $p$, we say that a finite power associative loop $Q$ has the \emph{Cauchy property for $p$} if whenever $p$ divides $|Q|$ then there is $x\in Q$ such that $|x|=p$. A finite power associative loop $Q$ is said to have the \emph{elementwise Lagrange property} if $|x|$ divides $|Q|$ for every $x\in Q$.

\begin{lem}\label{Lm:Coprime}
Let $Q$ be a finite power associative loop and let $d>1$.
\begin{enumerate}
\item[(i)] Suppose that $Q$ has the Cauchy property for every prime $p$ dividing $d$. If $Q$ is (uniquely) $d$-divisible then $|Q|$ is coprime to $d$.
\item[(ii)] Suppose that $Q$ has the elementwise Lagrange property. If $|Q|$ is coprime to $d$ then $Q$ is (uniquely) $d$-divisible.
\end{enumerate}
\end{lem}
\begin{proof}
(i) Suppose that $Q$ has the Cauchy property for every prime dividing $d$, and also assume that $|Q|$ is not coprime to $d$. Let $p$ be any common prime divisor of $d$ and $|Q|$. By assumption, there is $x\in Q$ such that $|x|=p$. By Lemma \ref{Lm:dPowers}, $Q$ is not uniquely $d$-divisible.

(ii) Suppose that $Q$ has the elementwise Lagrange property, and also assume that $Q$ is not uniquely $d$-divisible. By Lemma \ref{Lm:dPowers}, there is a prime $p$ dividing $d$ and some $x\in Q$ such that $|x|=p$. By assumption, $|x|$ divides $|Q|$, which implies that $|Q|$ is not coprime to $d$.
\end{proof}

\subsection{Moufang loops}\label{Ss:Moufang}

A loop $Q$ is \emph{Moufang} if it satisfies any one of the equivalent \emph{Moufang identities}
\begin{equation}\label{Eq:m1}
    xy\cdot zx = (x\cdot yz)x,\
    xy\cdot zx = x(yz\cdot x),\
    x(y\cdot zy) = (xy\cdot z)y,\
    x(y\cdot xz) = (xy\cdot x)z.
\end{equation}
We start by summarizing several well-known results for Moufang loops.

By Moufang Theorem \cite{Moufang,DrapalMT}, if three elements $x$, $y$ and $z$ of a Moufang loop associate, that is, $x(yz)=(xy)z$, then the subloop $\langle x,y,z\rangle$ is a group. Consequently, Moufang loops are diassociative, power associative, satisfy the \emph{flexible law} $x(yx)=(xy)x$, the \emph{inverse properties} $x^{-1}(xy)=y=(yx)x^{-1}$, and so on.

The four nuclei of a Moufang loop $Q$ coincide and form a normal subloop of $Q$.

All inner mappings of a Moufang loop can be seen as pseudoautomorphisms, with suitable companions. In particular,
\begin{equation}\label{Eq:InnPseudo}
    (x^{-3},T_x)
\end{equation}
is an element of $\lps(Q)$ in a Moufang loop $Q$. Moreover, every pseudoautomorphism of a Moufang loop is a semiautomorphism.

We proceed to less familiar results on Moufang loops.

\begin{lem}
Let $Q$ be a Moufang loop. Then
\begin{equation}\label{Eq:m2}
    x\m (xy\cdot z) = yx\m \cdot xz\quad \text{and}\quad (z\cdot yx)x\m = zx\cdot x\m y
\end{equation}
for every $x,y,z\in Q$.
\end{lem}
\begin{proof}
Note that $xy\cdot z = x(yx\m)x\cdot z = x(yx\m \cdot xz)$ by diassociativity and the Moufang identities \eqref{Eq:m1}. Multiplying by $x^{-1}$ on the left then yields the first identity. The second identity follows dually.
\end{proof}

\begin{lem}\label{Lm:m1}
Let $Q$ be a Moufang loop, $c\in Q$ and $f\in\sym(Q)$. Then $(c,f) \in \lps(Q)$ if and only if
\begin{displaymath}
    xc\m \cdot cy = f(f\m(x)f\m(y))
\end{displaymath}
for all $x,y \in Q$.
\end{lem}
\begin{proof}
Indeed, $cf(x)\cdot f(y) = cf(xy)$ if and only if $f(xy)=c\m(cf(x)\cdot f(y)) = f(x)c\m \cdot cf(y)$, by \eqref{Eq:m2}. We are done upon substituting $f^{-1}(x)$ for $x$ and $f^{-1}(y)$ for $y$.
\end{proof}

\begin{prop}\label{Pr:m2}
Let $Q$ be a Moufang loop. Then
\begin{displaymath}
    xa^{-3} \cdot a^3y = T_a\m(T_a(x)T_a(y))
\end{displaymath}
for all $a,x,y \in Q$.
\end{prop}
\begin{proof}
We have $(a^{-3},T_a)\in\lps(Q)$ by \eqref{Eq:InnPseudo}. By \eqref{Eq:LPs}, $(a^{-3},T_a)^{-1} = (T_a^{-1}(a^3),T_a^{-1}) = (a^3,T_a^{-1})$. We are done by Lemma \ref{Lm:m1}.
\end{proof}

\subsection{Lagrange and Cauchy properties for Moufang loops}

Finally, we present a few results on $d$-divisible Moufang loops, taking advantage of Lemma \ref{Lm:Coprime}.

It is not difficult to show that finite Moufang loops have the elementwise Lagrange property:

\begin{lem}\label{Lm:ElementwiseLagrange}
Let $Q$ be a finite power associative loop  satisfying the right power alternative identity $(ab^i)b^j = ab^{i+j}$ for every $i,j\in\mathbb Z$. Then $|x|$ divides $|Q|$ for every $x\in Q$.
\end{lem}
\begin{proof}
Let $x\in Q$ and $X=\langle x\rangle$. It suffices to show that the right cosets of $X$ partition $Q$. Suppose that $aX\cap bX\ne\emptyset$. Then $ax^i=bx^j$ for some $i,j\in\mathbb Z$, therefore $a = (bx^j)x^{-i} = bx^{j-i}$ by the right power alternative law, and thus $aX = (bx^{j-i})X = \{(bx^{j-i})x^k:k\in\mathbb Z\} = \{bx^{j-i+k}:k\in\mathbb Z\} = bX$.
\end{proof}

In general, Moufang loops do not satisfy the Cauchy property for every prime $p$. For instance, the smallest nonassociative simple Moufang loop of order $120$ contains no element of order $5$ \cite{Paige} and therefore it violates the Cauchy property for $p=5$. But the Cauchy property holds in Moufang loops for the primes $p=2$ and $p=3$:

\begin{thm}\label{Th:Cauchy}
Ever finite Moufang loop satisfies the Cauchy property for $p=2$ and $p=3$.
\end{thm}
\begin{proof}
For $p=2$, the standard group-theoretic argument works. Let $Q$ be a power associative loop of even order. The mapping $J:Q\to Q$, $x\mapsto x^{-1}$ is an involution and therefore has only orbits of sizes $1$ and $2$. Since $|Q|$ is even and $J(1)=1$, there must be $1\ne x\in Q$ such that $J(x)=x$, that is, $|x|=2$.

For $p=3$, consider the set $S=\{(x,y,z)\in Q\times Q\times Q:xy\cdot z=1\}$. For every $x,y\in Q$ there is a unique $z\in Q$ (namely $z=(xy)^{-1}$) such that $xy\cdot z = 1$. Hence $|S|=|Q|^2$. Moreover, if $xy\cdot z=1$ then $z\cdot xy=1$ and, by Moufang Theorem, $zx\cdot y=1$. Therefore $S$ is invariant under permuting its coordinates by $f=(1,2,3)$. Every orbit of $\langle f\rangle$ on $S$ has size $1$ or $3$, and $(x,y,z)\in S$ has orbit of size $1$ if and only if $x=y=z$. Since $3$ divides $|S|=|Q|^2$ and $(1,1,1)\in S$, there is $1\ne x\in Q$ such that $(x,x,x)=1$, which implies $|x|=3$.
\end{proof}

\begin{remark}
The elementary combinatorial proof of the Cauchy property for $p=3$ in Moufang loops is due to Doro (see \cite[Lemma 4]{Doro}). It is a variation on the well-known proof of the Cauchy property in groups for any prime $p$. In Section \ref{Sc:Cauchy3} we give another elementary but longer proof that is based on results of independent interest, cf.~Lemma \ref{Lm:Extend}.

Both Lemma \ref{Lm:ElementwiseLagrange} and Theorem \ref{Th:Cauchy} also follow from results in the theory of Moufang loops whose only known proofs depend on the classification of finite simple Moufang loops \cite{Liebeck} and hence also on the classification of finite simple groups. Lemma \ref{Lm:ElementwiseLagrange} is an immediate consequence of the Lagrange Theorem for Moufang loops, cf.~\cite{GagHalLagrange,GriZavLagrange}. For the Cauchy property, Grishkov and Zavarnitsine proved in \cite{GriZavSylow} that every Moufang loop of order $2^a3^bm$ with $m$ coprime to $6$ contains (Sylow) subloops of orders $2^a$ and $3^b$. Therefore, if $Q$ is a Moufang loop whose order is divisible by $p\in\{2,3\}$, it contains a subloop $X$ of order $p^c$ for some $c>0$, then any element $1\ne x\in X$ generates a cyclic group $\langle x\rangle$ of $p$-power order by Lemma \ref{Lm:ElementwiseLagrange}, and the group $\langle x\rangle$ then certainly contains an element of order $p$.
\end{remark}

Combining Lemma \ref{Lm:Coprime}, Lemma \ref{Lm:ElementwiseLagrange} and Theorem \ref{Th:Cauchy}, we get:

\begin{prop}\label{Pr:CoprimeUD}
Let $Q$ be a finite Moufang loop and let $d = 2^a3^b > 1$. Then $Q$ is uniquely $d$-divisible if and only if $|Q|$ is coprime to $d$.
\end{prop}

\section{Centrality, nilpotency, abelianess and solvability for loops}\label{Sc:Solvability}

In \cite{FM}, Freese and McKenzie developed commutator theory for congruence modular varieties. Their commutator of two congruences $\alpha$, $\beta$ in an algebra $Q$ will be denoted by $[\alpha,\beta]_Q$. The smallest congruence on $Q$ will be denoted by $\bot_Q = \{(x,x):x\in Q\}$ and the largest congruence on $Q$ will be denoted by $\top_Q = \{(x,y):x,y\in Q\}$.

The commutator theory of \cite{FM} was specialized to the variety of loops in \cite{StaVojComm}. Although we will not need to work with the exact form of the commutator of loop congruences (and instead take advantage of Theorem \ref{Th:AbExt} below), we state it here for the sake of completeness. In \cite[Theorem 2.1]{StaVojComm}, the commutator of loop congruences was expressed as the congruence generated by certain pairs of evaluated \emph{total} inner mappings. The technical complication with total inner mappings has been recently removed by Barnes who obtained the following description of the commutator of loop congruences in her PhD thesis \cite{Barnes}:

\begin{thm}[Barnes]\label{Th:CommCongr}
Let $\alpha$, $\beta$ be congruences of a loop $Q$. Then the commutator $[\alpha,\beta]_Q$ is the congruence of $Q$ generated by all pairs
\begin{displaymath}
    (T_{u_1}(a),T_{v_1}(a)),\quad (L_{u_1,u_2}(a),L_{v_1,v_2}(a)),\quad (R_{u_1,u_2}(a),R_{v_1,v_2}(a)),
\end{displaymath}
where $(1,a)\in \alpha$ and $(u_1,v_1)$, $(u_2,v_2)\in\beta$.
\end{thm}

In loops, there is a one-to-one correspondence between congruences and normal subloops. Given a normal subloop $X$ of a loop $Q$, the \emph{congruence $\alpha_X$ induced by $X$} is the equivalence relation on $Q$ with equivalence classes $\{aX:a\in Q\}$. Conversely, given a congruence $\alpha$ of a loop $Q$, the \emph{normal subloop of $Q$ induced by $\alpha$} is the equivalence class of $\alpha$ containing $1$.

We will therefore write $[X,Y]_Q$ for the \emph{commutator of normal subloops $X$, $Y$ of $Q$}, by which we mean the normal subloop of $Q$ induced by the congruence $[\alpha_X,\alpha_Y]_Q$. Theorem \ref{Th:CommCongr} can then be routinely translated to the context of normal subloops:

\begin{thm}
Let $X$, $Y$ be normal subloops of a loop $Q$. Then the commutator $[X,Y]_Q$ is the normal subloop of $Q$ generated by all quotients
\begin{displaymath}
    T_{u_1}(a)/T_{v_1}(a),\quad L_{u_1,u_2}(a)/L_{v_1,v_2}(a),\quad R_{u_1,u_2}(a)/R_{v_1,v_2}(a),
\end{displaymath}
where $a\in X$ and $u_1/v_1$, $u_2/v_2\in Y$.
\end{thm}

The commutator theory of \cite{FM} gives rise naturally to theories of central nilpotency and solvability. In loops, the central nilpotency theory of \cite{FM} coincides with the classical nilpotency theory adopted from groups. But the solvability theory of \cite{FM} is strictly stronger in loops than the classical solvability theory adopted from groups. Here are more details:

\subsection{Centrality and central nilpotency}

A congruence $\alpha$ of an algebra $Q$ is \emph{central} if $[\alpha,\top_Q]_Q=\bot_Q$. Passing to normal subloops, a normal subloop $X$ of a loop $Q$ is then said to be \emph{central} if $[X,Q]_Q=1$. Fortunately, this agrees with the traditional definition of centrality, because a normal subloop $X$ of $Q$ satisfies $[X,Q]_Q=1$ if and only if $X\le Z(Q)$.

\begin{definition}\label{Df:CenExt}
Given a commutative group $(X,+,0)$, a loop $(F,\cdot,1)$ and a mapping $\theta:F\times F\to X$ satisfying $\theta_{1,r} = 0 = \theta_{r,1}$ for every $r\in F$, the loop defined on $F\times X$ by
\begin{displaymath}
    (r,x)(s,y) = (rs,\ x+y+\theta_{r,s})
\end{displaymath}
is a \emph{central extension of $X$ by $F$}.
\end{definition}

\begin{thm}[{\cite[Theorem 4.2]{StaVojAbel}}]\label{Th:CenExt}
Let $X$ be a normal subloop of a loop $Q$. Then $X$ is central in $Q$ (that is, $[X,Q]_Q=1$) if and only if $Q$ is isomorphic to a central extension of $X$ by $Q/X$.
\end{thm}

A loop $Q$ is an \emph{iterated central extension} if it is either an abelian group or there exists a central subloop $X$ of $Q$ such that $Q/X$ is an iterated central extension.

\begin{definition}
A \emph{central series} for a loop $Q$ is a series
\begin{displaymath}
    Q=Q_0\ge Q_1\ge\cdots\ge Q_n=1,
\end{displaymath}
such that for every $0\le i<n$, $Q_{i+1}$ is a normal subloop of $Q$ and the factor $Q_i/Q_{i+1}$ is central in $Q/Q_{i+1}$ (that is, $[Q_i/Q_{i+1},Q/Q_{i+1}]_{Q/Q_{i+1}}=1$, or, equivalently, $Q_i/Q_{i+1}\le Z(Q/Q_{i+1}$)). A loop $Q$ is \emph{(centrally) nilpotent} if it contains a central series.
\end{definition}

\begin{thm}[{\cite[Corollary 5.2]{StaVojAbel}}]
A loop is centrally nilpotent if and only if it is an iterated central extension.
\end{thm}

\subsection{Abelianess}\label{Ss:Abelianess}

A congruence $\alpha$ of an algebra $Q$ is \emph{abelian} if $[\alpha,\alpha]_Q=\bot_Q$. An algebra $Q$ is \emph{abelian} if the congruence $\top_Q$ is abelian, that is, $[\top_Q,\top_Q]_Q = \bot_Q$.

A loop $Q$ is therefore \emph{abelian} if $[Q,Q]_Q=1$. It is well-known that a group is abelian if and only if it is a commutative group. More generally, a loop is abelian if and only if it is a commutative group, cf.~\cite{StaVojComm}.

A conflict in terminology arises when the adjective ``abelian'' is used for a normal subloop $X$ of a loop $Q$, since $X$ can be seen either as a congruence of $Q$ or as a loop in its own right. We will therefore be more careful in that context and say that a normal subloop $X$ of a loop $Q$ is \emph{abelian in $Q$} or that it \emph{induces an abelian congruence of $Q$} if $[X,X]_Q=1$, while we say that a normal subloop $X$ of a loop $Q$ is \emph{abelian} if $[X,X]_X=1$. Thus, for instance, the phrase ``$X$ is an abelian normal subloop of $Q$'' means that $X$ is a commutative group and $X$ is a normal subloop of $Q$.

Every normal subloop $X$ of $Q$ that induces an abelian congruence of $Q$ is an abelian normal subloop of $Q$. The converse is also true in the variety of groups, cf.~Lemma \ref{Lm:GrAb}, but there are numerous examples of loops $Q$ with an abelian normal subloop $X$ that does not induce an abelian congruence of $Q$. See \cite{StaVojComm} for examples of order $8$.

\begin{definition}\label{Df:AbExt}
Given a commutative group $(X,+,0)$, a loop $(F,\cdot,1)$ and mappings $\varphi,\psi:F\times F\to \aut(X)$ and $\theta:F\times F\to X$, the loop defined on $F\times X$ by
\begin{displaymath}
    (r,x)(s,y) = (rs,\ \varphi_{r,s}(x)+\psi_{r,s}(y)+\theta_{r,s})
\end{displaymath}
is an \emph{abelian extension of $X$ by $F$} if $\varphi_{r,1}=\id_X=\psi_{1,r}$ and $\theta_{1,r} = 0 = \theta_{r,1}$ for every $r\in F$.
\end{definition}

Central extensions are therefore those abelian extensions in which the automorphisms $\varphi_{r,s}$ and $\psi_{r,s}$ are trivial.

\begin{thm}[{\cite[Theorem 4.1]{StaVojAbel}}]\label{Th:AbExt}
Let $X$ be a normal subloop of a loop $Q$. Then $X$ is abelian in $Q$ (that is, $[X,X]_Q=1$) if and only if $Q$ is isomorphic to an abelian extension of $X$ by $Q/X$.
\end{thm}

The above external description of abelian extensions can be rewritten internally as follows. Let $(X,\cdot,1)$ be an abelian normal subloop of a loop $(Q,\cdot,1)$. Let $U$ be a (left) transversal to $X$ in $Q$ such that $1\in U$. Then $Q$ is an abelian extension of $X$ by $Q/X$ if there exist $\varphi$, $\psi:U\times U\to \aut(X)$ and $\theta:U\times U\to X$ satisfying $\varphi_{r,1}=\id_X=\psi_{1,r}$ and $\theta_{1,r}=1=\theta_{r,1}$ for every $r\in U$, and
\begin{equation}\label{Eq:i1}
    rx\cdot sy = u_{r,s}\cdot \varphi_{r,s}(x)\psi_{r,s}(y)\theta_{r,s}
\end{equation}
holds for every $r,s\in U$ and $x,y\in X$, where $u_{r,s}$ is the unique element of $U\cap (rs)X$.

There is a substantial difference between abelian normal subloops of $Q$ and normal subloops of $Q$ that are abelian in $Q$. This can be illustrated by considering the multiplication table of $Q$. Let $X$ be a commutative group. For $X$ to be an abelian normal subloop of $Q$, nothing else is required but that $Q$ is a disjoint union of $\{uX:u\in U\}$ for some subset $U\subseteq Q$, the multiplication table of $X$ is reproduced in the subsquare $X\times X$, and for every $r,s\in U$ the subsquare $rX \times sX$ is a latin square with entries running through $u_{r,s}X$, where $u_{r,s}\in U\cap (rs)X$. However, for $X$ to induce an abelian congruence of $Q$, the structure of the subsquare $rX \times sX$ must be much more rigid, conforming to \eqref{Eq:i1}.

Using the notion of abelian extension, it is easy to show that every abelian normal subgroup of a group $Q$ is abelian in $Q$. Here is a more general result:

\begin{lem}\label{Lm:GrAb}
Let $Q$ be a loop and let $X$ be an abelian normal subloop of $Q$ such that $X\le\nuc_m(Q)\cap\nuc_r(Q)$. Then $X$ induces an abelian congruence of $Q$.
\end{lem}
\begin{proof}
By Theorem \ref{Th:AbExt}, it suffices to show that $Q$ is an abelian extension of $X$ by $Q/X$. Let $U$ be a transversal to $X$ in $Q$. For $r,s\in U$, let $\varphi_{r,s}$ be the restriction of $T_s^{-1}$ to $X$, let $\psi_{r,s}=\id_X\in\aut(X)$, let $u_{r,s}\in U\cap (rs)X$ and let $\theta_{r,s} = u_{r,s}\ldiv (rs)\in X$. By Lemma \ref{Lm:NucAutT}, $\varphi_{r,s}\in\aut(X)$.

Let $x,y\in X$. Note that $sT_s^{-1}(x) = s(s\ldiv(xs)) = xs$. Since $x,y,T_s^{-1}(x)\in X\le \nuc_m(Q)\cap\nuc_r(Q)$, we have
\begin{align*}
	rx\cdot sy &= r(x\cdot sy) = r(xs\cdot y) = r(sT_s^{-1}(x)\cdot y) = r(s\cdot T_s^{-1}(x)y) = rs\cdot T_s^{-1}(x)y\\
	&= u_{r,s}\theta_{r,s}\cdot T_s^{-1}(x)y = u_{r,s}\cdot \theta_{r,s}(T_s^{-1}(x)y) =u_{r,s}\cdot T_s^{-1}(x)y\theta_{r,s},
\end{align*}
where the last step follows from the fact that $\theta_{r,s}$, $T_s^{-1}(x)$ and $y$ lie in the commutative group $X$. We have obtained an instance of \eqref{Eq:i1}, proving that $Q$ is an abelian extension of $X$ by $Q/X$.
\end{proof}

\subsection{Classical solvability and congruence solvability}

The history of the notion of solvability in loop theory is convoluted.

Albert defined solvable loops in \cite[p.~412]{AlbertII} as loops whose composition factors have no nontrivial subloops, mimicking a result from groups that states that a finite group is solvable if and only if each of its composition factors is isomorphic to a group of prime order. Albert's definition of solvability has been abandoned.

Bruck introduced the notion of a derived subloop in \cite[p.~268]{BruckTrans}. The \emph{derived subloop} $Q'$ of a loop $Q$ is the smallest normal subloop $H$ of $Q$ such that $Q/H$ is a commutative group. The \emph{derived series} of $Q$ is then the (possibly infinite) series
\begin{displaymath}
    Q=Q_0\ge Q_1\ge\cdots \ge Q_n\ge \cdots
\end{displaymath}
such that for every $i\ge 0$, $Q_{i+1}$ is the derived subloop of $Q_i$. Bruck then defines solvable loops as loops whose derived series reaches the trivial subloop $1$ in finitely many steps.

Another definition of solvability for loops was given by Glauberman. A loop $Q$ is said to be solvable in \cite[p.~397]{GlaubermanII} if there exists a series
\begin{displaymath}
    Q=Q_0\ge Q_1\ge\cdots\ge Q_n=1
\end{displaymath}
such that for every $0\le i<n$, $Q_{i+1}$ is a normal subloop of $Q_i$ and the factor $Q_i/Q_{i+1}$ is a commutative group. By adopting the standard proof from group theory, it is not difficult to show that Bruck's and Glauberman's definitions of solvability for loops are equivalent. In particular, Glauberman's definition of solvability will not be affected if we demand that for every $0\le i<n$, $Q_{i+1}$ is a normal subloop of $Q$, not just a normal subloop of $Q_i$.

The commutator theory of Freese and McKenzie \cite{FM} naturally leads to a definition of solvability in congruence modular varieties. In loops, their concept of solvability, called \emph{congruence solvability} in Definition \ref{Df:Solv}, is strictly stronger than the equivalent solvability concepts of Bruck and Glauberman, called \emph{classical solvability} in Definition \ref{Df:Solv}. The terminology comes from \cite{StaVojAbel}.

\begin{definition}\label{Df:Solv}
A \emph{classically solvable series} for a loop $Q$ is a series
\begin{displaymath}
    Q=Q_0\ge Q_1\ge\cdots\ge Q_n=1,
\end{displaymath}
such that for every $0\le i<n$, $Q_{i+1}$ is a normal subloop of $Q$ and the factor $Q_i/Q_{i+1}$ is abelian (that is, a commutative group). A loop $Q$ is \emph{classically solvable} if it contains a classically solvable series.

A \emph{congruence solvable series} for a loop $Q$ is a series
\begin{displaymath}
    Q=Q_0\ge Q_1\ge\cdots\ge Q_n=1
\end{displaymath}
such that for every $0\le i<n$, $Q_{i+1}$ is a normal subloop of $Q$ and the factor $Q_i/Q_{i+1}$ is abelian in $Q/Q_{i+1}$ (that is, $Q_i/Q_{i+1}$ induces an abelian congruence of $Q/Q_{i+1}$). A loop $Q$ is \emph{congruence solvable} if it contains a congruence solvable series.
\end{definition}

Obviously, every congruence solvable loop is classically solvable. The converse is true for groups but not for loops, with small counterexamples easy to construct, cf.~\cite{FM,StaVojComm}. The following problem is open:

\begin{prob}\label{Pr:MoufSolv}
Is every classically solvable Moufang loop congruence solvable?
\end{prob}

Towards a solution of Problem \ref{Pr:MoufSolv}, we prove in Sections \ref{Sc:Coincide} and \ref{Sc:Odd} that if $Q$ is a $6$-divisible classically solvable Moufang loop or a Moufang loop of odd order, then $Q$ is congruence solvable.

Since every central series is a congruence solvable series, we have:

\begin{thm}\label{Th:NilpSolv}
Centrally nilpotent loops are congruence solvable and classically solvable.
\end{thm}

Call a loop $Q$ an \emph{iterated abelian extension} if either $Q$ is a commutative group, or $Q$ is an abelian extension of a commutative group $X$ by some loop that is an iterated abelian extension. It was shown in \cite[Corollary 5.1]{StaVojAbel} that iterated abelian extensions of loops are precisely congruence solvable loops. For the sake of completeness, let us give a short proof here:

\begin{prop}\label{Pr:Solv}
A loop is congruence solvable if and only if it is an iterated abelian extension.
\end{prop}
\begin{proof}
Let $Q=Q_0>Q_1>\cdots>Q_n=1$ be a congruence solvable series for $Q$. We prove by induction on $n$ that $Q$ is an iterated abelian extension. If $n=1$ then the series becomes $Q>1$ and $Q$ is a commutative group, hence a congruence solvable loop. If $n>1$, let $X=Q_{n-1}\unlhd Q$ and note that $Q>X>1$. Since $X$ and $1$ are adjacent terms in the original series, $X/1=X$ is abelian in $Q/1=Q$. By Theorem \ref{Th:AbExt}, $Q$ is an abelian extension of $X$ by $Q/X$. It remains to show that $Q/X$ is an iterated abelian extension. Consider the series
\begin{equation}\label{Eq:Series}
    Q/X=Q_0/X>Q_1/X>\cdots>Q_{n-1}/X=X/X.
\end{equation}
Since $Q_i\unlhd Q$, we have $Q_i/X\unlhd Q/X$ by the Correspondence Theorem. As $Q_i/Q_{i+1}$ is abelian in $Q/Q_{i+1}$, the factor loop $(Q_i/X)/(Q_{i+1}/X)\cong Q_i/Q_{i+1}$ is abelian in $(Q/X)/(Q_{i+1}/X)\cong Q/Q_{i+1}$. Hence \eqref{Eq:Series} is a congruence solvable series, $Q/X$ is congruence solvable and, by the induction assumption, $Q/X$ is an iterated abelian extension.

Conversely, suppose that $Q$ is an iterated abelian extension constructed in $n$ steps, each an abelian extension. We prove by induction on $n$ that $Q$ is congruence solvable. If $n=0$ then $Q$ is a commutative group and we are done. Else let $Q$ be an abelian extension of a commutative group $X$ by $Q/X$, where $Q/X$ is constructed by $n-1$ abelian extensions. By Theorem \ref{Th:AbExt}, $X$ is abelian in $Q$. By the induction assumption, $Q/X$ is congruence solvable. Let $Q/X=Q_0/X>\cdots>Q_m/X=X/X$ be a congruence solvable series. Consider the series
\begin{equation}\label{Eq:Series2}
    Q=Q_0>Q_1>\cdots >Q_{m-1}>Q_m=X>1.
\end{equation}
Since $Q_i/X\unlhd Q/X$, we have $Q_i\unlhd Q$. Moreover, $Q_i/Q_{i+1}\cong (Q_i/X)/(Q_{i+1}/X)$ is abelian in $(Q/X)/(Q_{i+1}/X)\cong Q/Q_{i+1}$ for every $i<m$. The last inclusion in \eqref{Eq:Series2} is $X>1$, and certainly $X/1=X$ is abelian in $Q=Q/1$ as we have already shown. Therefore \eqref{Eq:Series2} is a congruence solvable series and $Q$ is congruence solvable.
\end{proof}

\section{Moufang loops with an abelian normal subgroup that does not induce an abelian congruence}\label{Sc:Examples}

By inspecting the library of small Moufang loops available in the \texttt{GAP} \cite{GAP} package \texttt{LOOPS} \cite{LOOPS}, it was observed in \cite{StaVojComm} that there exists a Moufang loop $Q$ of order $16$ with an abelian normal subloop $X$ (isomorphic to $C_2\times C_4$) such that $X$ does not induce an abelian congruence of $Q$.

In this section we offer a general construction of Moufang loops $Q$ containing an abelian normal subloop $X$ that does not induce an abelian congruence of $Q$.

Recall that for a vector space $V$ over a field $F$, a mapping $q:V\to F$ is a \emph{quadratic form} if $q(\lambda u) = \lambda^2u$ for every $\lambda\in F$, $u\in V$, and if $h:V\times V\to F$ defined by $h(u,v)= q(u+v)-q(u)-q(v)$ is a (symmetric) bilinear form. The form $h$ is referred to as the \emph{associated bilinear form}.

Also recall that a loop $Q$ is an \emph{extra loop} if it satisfies the identity $x(y\cdot zx) = (xy\cdot z)x$. Fenyves proved in \cite{Fenyves} that extra loops are Moufang.

\begin{prop}\label{Pr:Example}
Let $W=(W,+)$ be a commutative group with subgroups $F\le B\le W$. Suppose that $F=\{0,1\}$ and $\overline{W}=W/B$ is an elementary abelian $2$-group. Let $\overline q:\overline{W}\to F$ be a quadratic form with associated bilinear form $\overline h:\overline W\times \overline W\to F$. Let $q:W\to F$ and $h:W\times W\to F$ be defined by $q(u) = \overline q(\overline u)$ and $h(u,v) = \overline h(\overline u,\overline v)$. Denote by $Q=\mathcal Q(F,B,W,\overline q)$ the magma defined on $F\times W$ by
\begin{equation}\label{Eq:Example}
    (i,u)\cdot (j,v) = (i+j,\ u+v+jq(u)+ih(u,v)).
\end{equation}
Then:
\begin{enumerate}
\item[(i)] $Q$ is a centrally nilpotent loop, a central extension of the commutative group $B$ by the elementary abelian $2$-group $F\times \overline W$,
\item[(ii)] $Q$ is congruence solvable and hence classically solvable,
\item[(iii)] $Q$ is an extra loop,
\item[(iv)] $Q$ is a group if and only if the quadratic form $\overline q$ is linear,
\item[(v)] $X = 0\times W$ is an abelian normal subloop of $Q$,
\item[(vi)] if $Q$ is not a group, then the congruence of $Q$ induced by $X$ is not abelian.
\end{enumerate}
\end{prop}
\begin{proof}
(i) The commutative group $W$ is (isomorphic to) a central extension of $B$ by $\overline W$, with multiplication on $\overline W\times B$ given by
\begin{displaymath}
    (\overline u,a)(\overline v,b) = (\overline u+\overline v, a+b+\overline\theta_{\overline u,\overline v}),
\end{displaymath}
where $\overline\theta:\overline W\times \overline W\to B$ is some group cocycle. The multiplication formula \eqref{Eq:Example} on $F\times W$ can then be seen as a multiplication on $F\times\overline W\times B$, namely
\begin{displaymath}
    (i,\overline u,a)(j,\overline v,b) = (i+j,\overline u+\overline v, a+b+\overline\theta_{\overline u,\overline v}+j\overline q(\overline u)+i\overline h(\overline u,\overline v)).
\end{displaymath}
Hence $Q$ is isomorphic to the magma defined on $(F\times\overline W)\times B$ by
\begin{displaymath}
    ((i,\overline u),a)((j,\overline v),b) = ((i,\overline u)+(j,\overline v),a+b+\theta_{(i,\overline u),(j,\overline v)}),
\end{displaymath}
where $\theta_{(i,\overline u),(j,\overline v)} = \overline\theta_{\overline u,\overline v}+j\overline q(\overline u) + i\overline h(\overline u,\overline v)$. Since $\theta_{(0,\overline 0),(j,\overline v)} = \overline\theta_{\overline 0,\overline v} + j\overline q(\overline 0) = 0$ and $\theta_{(i,\overline u),(0,\overline 0)} = \overline\theta_{\overline u,\overline 0} + i\overline h(\overline u,\overline 0) = 0$, we see at once that $\theta$ is a loop cocycle and $Q$ is a central extension of the commutative group $B$ by the elementary abelian $2$-group $F\times\overline W$. Hence $Q$ is a centrally nilpotent loop, finishing the proof of part (i). Part (ii) then follows from Theorem \ref{Th:NilpSolv}.

(iii) Chein an Robinson characterized extra loops as Moufang loops in which squares are in the nucleus \cite{CheinRobinson}. In the isomorphic copy of $Q$ from part (i) we have $((i,\overline u),a)^2 = ((0,\overline 0),2a+\theta_{(i,\overline u),(i,\overline u)})\in 0\times B\le Z(Q)\le\nuc(Q)$. To prove that $Q$ is extra, it therefore suffices to check one of the Moufang identities, say $(xy\cdot x)z = x(y\cdot xz)$. By definition, $q(b+v)=q(v)$ and $h(b+v,w) = h(v,w)$ for every $b\in B$ and $v,w\in W$. Note that $2W\subseteq B$ since $\overline W = W/B$ is an elementary abelian $2$-group. Finally observe that $h(u,u) = \overline h(\overline u,\overline u) = \overline q(\overline 0)+2\overline q(\overline u) = 0$ and $h(u,u+v) = h(u,u)+h(u,v) = h(u,v)$ for every $u,v\in Q$. Let $x=(i,u)$, $y=(j,v)$ and $z=(k,w)\in F\times W$. We then calculate
\begin{align*}
    &((i,u)(j,v)\cdot (i,u))(k,w)\\
    &\quad = (i+j,u+v+jq(u)+ih(u,v))(i,u)\cdot (k,w)\\
    &\quad = (j,2u+v+jq(u)+ih(u,v)+iq(u+v)+(i+j)h(v,u))\cdot (k,w)\\
    &\quad = (j,2u+v+jq(u)+iq(u+v)+jh(u,v))\cdot (k,w)\\
    &\quad = (j+k,2u+v+w+jq(u)+iq(u+v)+jh(u,v)+kq(v)+jh(v,w)).
\end{align*}
On the other hand
\begin{align*}
    &(i,u)((j,v)\cdot (i,u)(k,w))\\
    &\quad = (i,u)\cdot (j,v)(i+k,u+w+kq(u)+ih(u,w))\\
    &\quad = (i,u)\cdot (i+j+k,u+v+w+kq(u)+ih(u,w)+(i+k)q(v)+jh(v,u+w))\\
    &\quad = (j+k,2u{+}v{+}w{+}kq(u){+}ih(u,w){+}(i{+}k)q(v){+}jh(v,u{+}w){+}(i{+}j{+}k)q(u){+}ih(u,v{+}w))\\
    &\quad = (j+k,2u+v+w+(i+j)q(u)+ih(u,v)+(i+k)q(v)+jh(v,u+w)).
\end{align*}
These two products agree in the first coordinate. Upon canceling like terms in the second coordinate (while taking advantage of bilinearity of $h$), only $iq(u+v)$ remains in the first product, while $iq(u)+ih(u,v)+iq(v)$ remains in the second product. Since $h(u,v)=q(u+v)+q(u)+q(v)$, we are done.

(iv) Suppose that $\overline q$ is linear. Then $\overline h=0$ and the cocycle of (i) reduces to $\theta_{(i,\overline u),(j,\overline v)}= \overline\theta_{\overline u,\overline v}+j\overline q(\overline u)$. Since $\overline\theta$ is a group cocycle, it satisfies the group cocycle identity
\begin{displaymath}
    \overline\theta_{\overline u,\overline v}+\overline\theta_{\overline u+\overline v,\overline w} = \overline\theta_{\overline v,\overline w} + \overline\theta_{\overline u,\overline v+\overline w}.
\end{displaymath}
We then have
\begin{align*}
    \theta_{(i,\overline u),(j,\overline v)} + \theta_{(i,\overline u)+(j,\overline v),(k,\overline w)}
       & = \overline\theta_{\overline u,\overline v} + \overline\theta_{\overline u+\overline v,\overline w} + j\overline q(\overline u) + k\overline q(\overline u+\overline v)\\
       & = \overline\theta_{\overline v,\overline w} + \overline\theta_{\overline u,\overline v+\overline w} + k\overline q(\overline v) + (j+k)\overline q(\overline u)\\
       & = \theta_{(j,\overline v),(k,\overline w)} + \theta_{(i,\overline u),(j,\overline v)+(k,\overline w)},
\end{align*}
which is a group cocycle identity for $\theta$, so $Q$ is a group.

Conversely, if $Q$ is a group then for all $u,v\in W$ the product
\begin{displaymath}
    (1,0)(0,u)\cdot (1,v) = (1,u)(1,v) = (0,u+v+q(u)+h(u,v))
\end{displaymath}
is equal to
\begin{displaymath}
    (1,0)\cdot (0,u)(1,v) = (1,0)(1,u+v+q(u)) = (0,u+v+q(u)),
\end{displaymath}
so $h=0$ and $q$ is linear.

(v) On $X=0\times W$, the multiplication formula \eqref{Eq:Example} reduces to
\begin{displaymath}
	(0,u)(0,v) = (0,u+v),
\end{displaymath}
and $X$ is therefore isomorphic to the commutative group $W$. The mapping $f:Q\to F$, $(i,u)\mapsto i$ is clearly a homomorphism with kernel equal to $X$, which shows that $X$ is a normal subloop of $Q$.

(vi) Suppose that $X$ induces an abelian congruence of $Q$. By Theorem \ref{Th:AbExt}, $Q$ is an abelian extension of $X=W$ by $F$, and there exist $\varphi,\psi:F\times F\to\aut(X)$ and $\theta:F\times F\to X$ such that $\varphi_{i,0}=\psi_{0,i}=\id_X$, $\theta_{i,0}=\theta_{0,i}=0$ and
\begin{displaymath}
	(i,u)(j,v) = (i+j,\varphi_{i,j}(u)+\psi_{i,j}(v)+\theta_{i,j}).
\end{displaymath}
Since $(0,u)(1,v) = (1,u+v+q(u))$ by \eqref{Eq:Example} and $\varphi_{0,1}(u) + \psi_{0,1}(v)+\theta_{0,1} = \varphi_{0,1}(u) + v$, we must have $\varphi_{0,1}(u)=u+q(u)$. As $\varphi_{0,1}$ is an automorphism of $X$, we deduce $(u+v)+q(u+v) = (u+q(u))+(v+q(v))$ for all $u,v\in W$. This shows that $q$ is linear and $Q$ is a group by (iii).
\end{proof}

The nonassociative loops $\mathcal Q(F,B,W,\overline q)$ afforded by Proposition \ref{Pr:Example} demonstrate that it is possible for a Moufang loop to be congruence solvable (even centrally nilpotent) and yet posses an abelian normal subloop that does not induce an abelian congruence.

\begin{example}
Smallest examples of interest are obtained from Proposition \ref{Pr:Example} when $\overline q$ is a nonlinear quadratic form on a vector space $W/B$ of dimension two, which forces (up to isomorphism) either $F=B=0\times 0\times  C_2\le W= C_2\times C_2\times C_2$ or $F=B=0\times  C_2\le W= C_2\times C_4$. We then obtain a Moufang loop $Q$ of order $16$ with an abelian normal subloop $X$ (isomorphic to $ C_2\times C_2\times C_2$ or to $ C_2\times C_4$) that does not induce an abelian congruence of $Q$. This covers the example that was found in \cite{StaVojComm} by an exhaustive search.
\end{example}

\section{Moufang loops in which classical solvability and congruence solvability coincide}\label{Sc:Coincide}

We start with the following easy fact:

\begin{lem}\label{Lm:Comm2Div}
Let $X$ be a $2$-divisible commutative group. Then every semiautomorphism of $X$ is an automorphism of $X$.
\end{lem}
\begin{proof}
Let $f$ be a semiautomorphism of $X$. Recall that $f(x^n)=f(x)^n$ for every $x\in X$ and $n\in\mathbb Z$. Let $x,y\in X$. By $2$-divisibility, there is $u\in X$ such that $x=u^2$. Then $f(xy) = f(u^2y) = f(uyu) = f(u)f(y)f(u) = f(u)^2f(y) = f(u^2)f(y) = f(x)f(y)$.
\end{proof}

\begin{lem}\label{Lm:Better}
Let $Q$ be a Moufang loop and $X$ a $2$-divisible abelian normal subgroup of $Q$. Then every inner mapping of $Q$ restricts to an automorphism of $X$.
\end{lem}
\begin{proof}
Let $f\in\inn(Q)$. By the introductory remarks in Subsection \ref{Ss:Moufang}, $f$ is a pseudoautomorphism of $Q$ and hence a semiautomorphism of $Q$. By Lemma \ref{Lm:Comm2Div}, the restriction $f|_X$ of $f$ to $X$ is an automorphism of $X$.
\end{proof}

\begin{thm}\label{Th:AbelianCongruence}
Let $Q$ be a $3$-divisible Moufang loop and let $X$ be a $2$-divisible abelian normal subgroup of $Q$. Then the congruence $\{aX:a\in Q\}$ on $Q$ induced by $X$ is an abelian congruence of $Q$.
\end{thm}
\begin{proof}
Let $U$ be a transversal to $X$ containing $1$, let $r,s\in U$ and $x,y\in X$. There are uniquely determined $u=u_{r,s}\in U$ and $z\in X$ such that $rs=uz$. We wish to apply Theorem \ref{Th:AbExt} and hence to find $\varphi_{r,s},\psi_{r,s}\in\aut(X)$ and $\theta_{r,s}\in X$ such that $rx\cdot sy = u\cdot \varphi_{r,s}(x)\psi_{r,s}(y)\theta_{r,s}$ as in \eqref{Eq:i1}. In addition, we must verify $\varphi_{r,1}=\id_X = \psi_{1,r}$ and $\theta_{1,r}=\theta_{r,1}=1$. We will build the automorphisms in several steps. Consider
\begin{displaymath}
    f_1 = (T_s)|_X.
\end{displaymath}
By Lemma \ref{Lm:Better}, $f_1\in\aut(X)$. We have $f_1(y)s = T_s(y)s = sys^{-1}s=sy$, so
\begin{displaymath}
    rx\cdot sy = rx\cdot f_1(y)s.
\end{displaymath}
Let
\begin{displaymath}
    f_2 = (L_{s^{-1}r}^{-1}L_s^{-1}L_r)|_X.
\end{displaymath}
Since $f_2(1) = (s^{-1}r)^{-1}(s^{-1}r)=1$, $f_2$ is a restriction of an inner mapping of $Q$ to $X$. By Lemma \ref{Lm:Better}, $f_2\in\aut(X)$. Moreover, we have $s\cdot (s^{-1}r)f_2(x)=rx$. Therefore
\begin{displaymath}
    rx\cdot sy = rx\cdot f_1(y)s = (s\cdot (s^{-1}r)f_2(x))(f_1(y)s) = s( (s^{-1}r)f_2(x)\cdot f_1(y))s,
\end{displaymath}
where we have used a Moufang identity \eqref{Eq:m1} in the last step. By the identity \eqref{Eq:m2}, we have $uv\cdot w = u(u^{-1}\cdot(uv\cdot w)) = u(vu^{-1}\cdot uw)$ for any $u,v,w\in Q$. In particular, with $u=s^{-1}r$, $v=f_2(x)\in X$ and $w=f_1(y)\in X$, we obtain
\begin{displaymath}
    rx\cdot sy = s( uv\cdot w )s = s(u(vu^{-1}\cdot uw))s = su\cdot (vu^{-1}\cdot uw)s,
\end{displaymath}
where we have again used a Moufang identity in the last step. Since $Q$ is $3$-divisible, there is $a\in Q$ such that $u=a^3$. Proposition \ref{Pr:m2} then yields
\begin{displaymath}
    vu^{-1}\cdot uw = va^{-3}\cdot a^3w = T_a^{-1}(T_a(v)T_a(w)) = vw,
\end{displaymath}
where in the last step we used $(T_a)|_X\in\aut(X)$, by Lemma \ref{Lm:Better}. So far we showed
\begin{displaymath}
    rx\cdot sy = su\cdot (vw)s =  s(s^{-1}r)\cdot (vw)s = r\cdot (vw)s = r\cdot (f_2(x)f_1(y))s.
\end{displaymath}
Now, $sf_1^{-1}(x_0) = ss^{-1}x_0s = x_0s$ for any $x_0\in Q$ and thus
\begin{displaymath}
    rx\cdot sy = r\cdot (f_2(x)f_1(y))s = r\cdot sf_1^{-1}(f_2(x)f_1(y)) = r\cdot s(f_1^{-1}f_2(x)\cdot y),
\end{displaymath}
taking advantage of $f_1\in\aut(X)$. Consider
\begin{displaymath}
    f_3 = (L_{rs}^{-1}L_rL_s)|_X
\end{displaymath}
and note that $f_3(1) = (rs)^{-1}(rs)=1$, which implies $f_3\in\aut(X)$ as usual. Moreover, $rs\cdot f_3(x_0) = r\cdot sx_0$ for any $x_0\in Q$, and hence
\begin{displaymath}
    rx\cdot sy = r\cdot s(f_1^{-1}f_2(x)\cdot y) = rs\cdot f_3(f_1^{-1}f_2(x)\cdot y) = rs\cdot (f_3f_1^{-1}f_2(x)\cdot f_3(y)),
\end{displaymath}
using $f_3\in\aut(X)$. Recall that $rs=uz$ with $u\in U$, $z\in X$, and consider
\begin{displaymath}
    f_4 = (L_z^{-1}L_u^{-1}L_{rs})|_X.
\end{displaymath}
We have $f_4(1) = z^{-1}(u^{-1}(rs)) = 1$ (since $rs= uz$) and $f_4\in\aut(X)$. Moreover, $u\cdot zf_4(x_0) = rs\cdot x_0$ for any $x_0\in Q$, and thus
\begin{displaymath}
    rx\cdot sy{=}rs\cdot (f_3f_1^{-1}f_2(x)\cdot f_3(y)){=}u\cdot zf_4(f_3f_1^{-1}f_2(x)\cdot f_3(y))
	{=}u\cdot f_4f_3f_1^{-1}f_2(x)f_4f_3(y)z,
\end{displaymath}
where we have used $f_4\in\aut(X)$ and commutativity of the group $X$ in the last step. Rewriting, we have
\begin{displaymath}
    rx\cdot sy = u\cdot \varphi_{r,s}(x)\psi_{r,s}(y)\theta_{r,s},
\end{displaymath}
where $\varphi_{r,s} = f_4f_3f_1^{-1}f_2\in\aut(X)$, $\psi_{r,s} = f_4f_3\in\aut(X)$ and $\theta_{r,s} = z = u^{-1}(rs)\in X$.

If $s=1$, we observe $f_1=(T_1)|_X = \id_X$, $f_2 = (L_r^{-1}L_1^{-1}L_r)|_X = \id_X$, $f_3 = (L_r^{-1}L_rL_1)|_X = \id_X$, $u\in (r1)X\cap U = \{r\}$, $\theta_{r,1}=z=1$, $f_4 = (L_1^{-1}L_r^{-1}L_r)|_X = \id_X$ and therefore $\varphi_{r,1}=\id$. If $r=1$, we observe $f_3 = (L_s^{-1}L_1L_s)|_X = \id_X$, $u\in (1s)X\cap U = \{s\}$, $\theta_{1,s}=z=1$, $f_4 = (L_1^{-1}L_s^{-1}L_s)|_X = \id_X$ and therefore $\psi_{1,s}=\id_X$.
\end{proof}

Recall the power mapping $h_d$ of \eqref{Eq:dPower} and note that $h_6=h_3h_2=h_2h_3$ in a power associative loop. Hence a power associative loop is $6$-divisible if and only if it is $2$-divisible and $3$-divisible. We therefore deduce from Theorem \ref{Th:AbelianCongruence}:

\begin{cor}\label{Cr:AbelianCongruence}
Let $Q$ be a $6$-divisible Moufang loop and let $X$ be an abelian normal subloop of $Q$. Then $X$ induces an abelian congruence of $Q$.
\end{cor}

Since in the two definitions of solvability (Definition \ref{Df:Solv}), the only difference is whether the quotients $Q_i/Q_{i+1}$ are merely commutative groups or whether they induce an abelian congruence, we have:

\begin{cor}\label{Cr:Coincide}
Let $Q$ be a $6$-divisible Moufang loop. Then $Q$ is solvable if and only if it is congruence solvable.
\end{cor}

\section{Moufang loops of odd order are congruence solvable}\label{Sc:Odd}

We proceed to show that Moufang loops of odd order are congruence solvable. This strengthens Glauberman's Odd Order Theorem for Moufang loops \cite[Theorem 16]{GlaubermanII}:

\begin{thm}[Glauberman]\label{Th:Glauberman}
A Moufang loop of odd order is classically solvable.
\end{thm}

Our proof is based on the recent, deep result of Cs\"org\H{o} \cite[Theorem 5.1]{Csorgo}:

\begin{thm}[Cs\"org\H{o}]\label{Th:Csorgo}
A nontrivial Moufang loop of odd order has a nontrivial nucleus.
\end{thm}

A subloop $X$ of a loop $Q$ is \emph{characteristic} if $f(X)=X$ for every $f\in\aut(Q)$. Unlike in groups, a characteristic subloop of $Q$ is not necessarily a normal subloop of $Q$. It is clear from the definition of the nucleus that $\nuc(Q)$ is a characteristic subloop of $Q$.

\begin{thm}\label{Th:MoufOdd}
A Moufang loop of odd order is congruence solvable.
\end{thm}
\begin{proof}
We proceed by induction on the order of the Moufang loop $Q$. There is nothing to prove when $Q=1$, so suppose that $Q$ is nontrivial. In every Moufang loop, the nucleus is a normal subloop. By Theorem \ref{Th:Csorgo}, $1<N=\nuc(Q)\unlhd Q$. Since $N$ is a group of odd order, it is solvable by the Odd Order Theorem for groups \cite{FeitThompson}.

Let $X$ be a minimal characteristic subgroup of $N$. Let $f\in\inn(Q)$. By Corollary \ref{Cr:NucAut}, $f$ restricts to an automorphism of $N$. Since $X$ is characteristic in $N$, $f(X)=X$. Hence $X\unlhd Q$. A variation on a standard group-theoretic argument, cf.~\cite[Theorem 5.24]{Rotman}, now shows that $X$ is an abelian group. (Consider the derived subgroup $X'$ of the solvable group $X$.)

Altogether, we have shown that $X$ is an abelian normal subgroup of $Q$ and $1<X\le\nuc(Q)$. By Lemma \ref{Lm:GrAb}, $X$ induces an abelian congruence of $Q$. By Theorem \ref{Th:AbExt}, $Q$ is an abelian extension of $X$ by $Q/X$. By the induction assumption, $Q/X$ is congruence solvable. By Proposition \ref{Pr:Solv}, $Q/X$ is an iterated abelian extension, hence $Q$ is an iterated abelian extension, and $Q$ is congruence solvable by Proposition \ref{Pr:Solv} again.
\end{proof}

Note that Theorem \ref{Th:MoufOdd} implies the finitary version of Corollary \ref{Cr:Coincide}. Indeed, if $Q$ is a finite $6$-divisible Moufang loop then it is of odd order by Proposition \ref{Pr:CoprimeUD}, congruence solvable by Theorem \ref{Th:MoufOdd} and thus also classically solvable.

\section{The Cauchy property for $p=3$ in Moufang loops}\label{Sc:Cauchy3}

Let us give another elementary proof of the Cauchy property for $p=3$ in Moufang loops. We are mostly interested in some of the intermediate results, e.g., Lemma \ref{Lm:Extend}. The ideas presented here overlap substantially with those of Doro \cite{Doro}, Glauberman \cite{GlaubermanII} and Hall \cite[Chapter 13]{Hall}.

We start with a lemma motivated by triality for Moufang loops. Let
\begin{displaymath}
	\com(Q)=\{x\in Q:xy=yx\text{ for all }y\in Q\}
\end{displaymath}
be the \emph{commutant} of $Q$. In addition to the already introduced bijections $L_x$, $R_x$ and $T_x = R_x^{-1}L_x$, consider also $M_x = R_xL_x$.  (Glauberman uses $P_x$ for $R_xL_x$, while Doro and Hall use $P_x$ for $R_x^{-1}L_x^{-1}$, so it seems prudent to introduce $M_x$ to avoid confusion.) In diassociative loops, we have $L_x^{n} = L_{x^n}$, $M_x = L_xR_x$, etc.

\begin{lem}\label{Lm:TrialityMap}
Let $Q$ be a diassociative loop. Then a mapping
\begin{displaymath}
	\rho:\{L_x,R_x,L_x^{-1},R_x^{-1}:x\in Q\}\to \mlt(Q)
\end{displaymath}
is well-defined by $\rho(L_x)=R_x$, $\rho(L_x^{-1}) = R_x^{-1}$, $\rho(R_x) = M_x^{-1}$ and $\rho(R_x^{-1}) = M_x$ if and only if $x^3=1$ for every $x\in\com(Q)$.
\end{lem}
\begin{proof}
Note that if $L_x\in\{L_y,R_y,L_y^{-1},R_y^{-1}\}$ then $x\in\{y,y^{-1}\}$, and similarly for $R_x$. To check that $\rho$ is well-defined, we therefore need to establish the following implications: (a) $L_x=L_x^{-1}$ implies $R_x = R_x^{-1}$, (b) $R_x=R_x^{-1}$ implies $M_x^{-1}=M_x$, (c) $L_x=R_x$ implies $R_x=M_x^{-1}$, (d) $L_x=R_x^{-1}$ implies $R_x = M_x$, (e) $L_x^{-1} = R_x$ implies $R_x^{-1}=M_x^{-1}$, and (f) $L_x^{-1}=R_x^{-1}$ implies $R_x^{-1}=M_x$.

The implication (a) is immediate, for if $L_x=L_x^{-1}$ then $x=x^{-1}$ and $R_x = R_x^{-1}$. The argument for (b) is similar. The implications (c) and (f) are equivalent, and so are the implications (d) and (e).

Concerning (c), suppose that $L_x=R_x$, that is, $x\in\com(Q)$. We will show that then $R_x = M_x^{-1}$ iff $x^3=1$. Certainly if $R_x=M_x^{-1}$ then evaluating at $1$ yields $x=x^{-2}$, that is, $x^3=1$. Conversely, if $x^3=1$ then $R_x^3=\id_Q$ and $R_x = R_x^{-2} = R_x^{-1}L_x^{-1} = M_x^{-1}$.

Finally, for (d), suppose that $L_x=R_x^{-1}$, that is, $x^2=1$ and $x\in\com(Q)$. We will again show that $R_x=M_x$ iff $x^3=1$. Certainly if $R_x=M_x$ then $R_x=R_xL_x$, $L_x=\id_Q$, $x=1$ and $x^3=1$. Conversely, if $x^3=1$ then from $x^2=1$ we deduce $x=1$ and $R_x=M_x$.
\end{proof}

\begin{cor}\label{Cr:TrialityMap}
If $Q$ is a Moufang loop with trivial nucleus then the mapping $\rho$ of Lemma \emph{\ref{Lm:TrialityMap}} is well-defined.
\end{cor}
\begin{proof}
Let $x\in\com(Q)$. Then $T_x=\id_Q$, so in particular $T_x$ is an automorphism of $Q$. Since $(x^{-3},T_x)\in\lps(Q)$ by \eqref{Eq:InnPseudo}, we have $x^{-3}T_x(y)\cdot T_x(z) = x^{-3}T_x(yz) = x^{-3}(T_x(y)T_x(z))$ for all $y,z\in Q$. Therefore $x^{-3}\in\nuc(Q)=1$.
\end{proof}

\begin{lem}\label{Lm:Extend}
Let $Q$ be a Moufang loop with trivial nucleus. Then there exists a unique $\rho\in\aut(\mlt(Q))$ such that $\rho(L_x) = R_x$ and $\rho(R_x) = M_x\m$, for every $x\in Q$.
This automorphism satisfies $\rho^3=\id_{\mlt(Q)}$, and if $\vhi$ is an inner mapping of $Q$ with companion $c$ (when seen as a pseudoautomorphism) then $\rho(\vhi) = R_c^{-1}\vhi$.
\end{lem}

\begin{proof}
The first part is proved already in \cite[Theorem 6]{GlaubermanII} but let us give a proof. The mapping $\rho$ of Lemma \ref{Lm:TrialityMap} is well-defined by Corollary \ref{Cr:TrialityMap}. Our first goal is to show that it is possible to extend it into an automorphism of $\mlt(Q)$. For that it suffices to verify that
\begin{equation}\label{Eq:t3}
	\psi_1\cdots \psi_k = \id_Q \quad \Leftrightarrow \quad \rho(\psi_1)\cdots \rho(\psi_k) = \id_Q
\end{equation}
whenever each $\psi_i$, $1\le i \le k$, belongs to $\{L_x^{\pm 1},R_x^{\pm 1}:x\in Q\}$.

The first Moufang identity of \eqref{Eq:m1} can be rewritten as $(L_x,R_x,M_x)\in\atp(Q)$. Substituting $yx^{-1}$ for $y$ in the second identity of \eqref{Eq:m2} yields $(zy)x^{-1} = zx\cdot x^{-1}(yx^{-1})$, which says $(R_x,M_x^{-1},R_x^{-1})\in\atp(Q)$. Taking inverses into consideration, we see that all four triples
\begin{equation}\label{Eq:t4}
	(L_x,R_x,M_x),\,(L_x\m,R_x\m,M_x\m),\,(R_x,M_x\m,R_x\m) \text{ and } (R_x\m,M_x,R_x)
\end{equation}
are autotopisms of $Q$, for every $x\in Q$.

Now, autotopisms may be chosen from this list in such a way that the first coordinate is equal to $\psi_i$ and the second coordinate is equal to $\rho(\psi_i)$, $1\le i \le k$. By composing these autotopisms we obtain an autotopism in which the first coordinate is equal to $\psi_1\cdots \psi_k$ and the second coordinate is equal to $\rho(\psi_1)\cdots \rho(\psi_k)$. Since all the (equal) nuclei of $Q$ are trivial, Corollary \ref{Cr:TrivNucAtp} implies that the first coordinate is trivial if and only if the second coordinate is trivial. We have proved $\rho\in\aut(\mlt(Q))$.

Let us establish $\rho^3=\id_{\mlt(Q)}$. We have $\rho(L_x) = R_x$, $\rho(R_x) = M_x^{-1} = R_x^{-1}L_x^{-1}$ and $\rho(M_x^{-1}) = \rho(R_x^{-1})\rho(L_x^{-1}) = M_xR_x^{-1} = L_xR_xR_x^{-1} = L_x$. This shows that the automorphism $\rho^3$ is identical on a generating set of $\mlt(Q)$, hence also on $\mlt(Q)$.

Finally, let $\vhi\in\inn(Q)$. Then $\vhi$ is a pseudoautomorphism with some companion $c\in Q$. This can also be expressed as $(L_c\vhi,\vhi,L_c\vhi)\in\atp(Q)$. Let us compose the autotopisms of \eqref{Eq:t4} so that the first coordinate of the resulting autotopism $(f,g,h)$ is equal to $L_c\vhi$. Certainly $g=\rho(f)$. By Corollary \ref{Cr:TrivNucAtp}, $(f,g,h)=(L_c\vhi,\vhi,L_c\vhi)$. Thus $R_c\rho(\vhi) = \rho(L_c)\rho(\vhi) = \rho(L_c\vhi)  = \rho(f) = g = \vhi$ and $\rho(\vhi) = R_c^{-1}\vhi$ follows.
\end{proof}

\begin{prop}\label{Pr:Proper}
Let $Q$ be a finite Moufang loop of order divisible by three. If the nucleus of $Q$ is trivial, then $Q$ possesses a proper subloop of order divisible by three.
\end{prop}
\begin{proof}
Let $\rho$ be as in Lemma \ref{Lm:Extend}, $H=\langle\rho\rangle$, and let $G=\mlt(Q)\rtimes H$ be the semidirect product defined by the natural action of $H$ on $\mlt(Q)$. Hence $\lambda \rho^i \cdot \mu \rho^j = \lambda\rho^i(\mu)\cdot  \rho^{i+j}$ for all $\lambda,\mu \in \mlt(Q)$ and $i,j\in\mathbb Z$.

By Lemma \ref{Lm:Extend}, $\rho^3=\id_{\mlt(Q)}$. Suppose that $\rho = \id_{\mlt(Q)}$. Then $L_x = \rho(L_x) = R_x$, $Q$ is a commutative Moufang loop of exponent $3$ (by Corollary \ref{Cr:TrialityMap} and Lemma \ref{Lm:TrialityMap}, or see \cite[Lemma 2]{Doro}), and any $1\ne x\in Q$ gives rise to a subloop of order $3$, necessarily proper since $Q$ is not a group. We can therefore suppose that $|H|=3$. Consider a 3-Sylow subgroup $S$ of $G$ that contains $H$. Let $P=\mlt(Q)\cap S$. As $S$ contains $H$, the underlying set of $S$ is equal to $P\times H$. Since $|\mlt(Q)|=|Q|\cdot|\inn(Q)|$ in any loop, and since $3$ divides $|Q|$ here, it follows that $3$ divides $|\mlt(Q)|$. Then $|S|>3$ and $P>1$. As $\mlt(Q)$ is normal in $G$, the group $P=\mlt(Q)\cap S$ is normal in $S$. The center $Z(P)\ne 1$ of $P$ is also normal in $S$, being a characteristic subgroup of $P$. The automorphism $\rho$ acts by conjugation on $Z(P)$. Since $\rho^3=\id_{\mlt(Q)}$ and $Z(P)\ne 1$ is a $3$-group, the conjugation by $\rho$ also fixes some $1\ne\psi\in Z(P)$. Passing to a suitable power of $\psi$, we can assume that $|\psi|=3$. Note that $\psi\rho = \rho\psi = \rho(\psi)\rho$ implies $\rho(\psi)=\psi$.

Write $\psi=L_x\vhi$ for some $x\in Q$ and $\vhi \in\inn(Q)$. Let $c$ be the companion of $\vhi$. By Lemma \ref{Lm:Extend}, $L_x\vhi = \psi = \rho(\psi) = \rho(L_x\vhi) = \rho(L_x)\rho(\vhi) = R_xR_c^{-1}\vhi$, hence $L_x=R_xR_c^{-1}$, $c=1$ (so $\vhi$ is an automorphism), $L_x=R_x$ and $x\in\com(Q)$. By Corollary \ref{Cr:TrialityMap} and Lemma \ref{Lm:TrialityMap}, $x^3=1$. If $x\ne 1$ then $\langle x\rangle$ is the sought-after subloop. Else $x=1$ and $\psi=\vhi$ is an automorphism of order $3$. Since $3$ divides $|Q|$, it then also divides the order of the proper subloop $\mathrm{Fix}(\vhi) = \{u\in Q:\vhi(u) = u\}$.
\end{proof}

We are ready to prove the Cauchy property for $p=3$ in Moufang loops. Let $Q$ be a finite Moufang loop whose order is divisible by $3$. We proceed by induction on $|Q|$. Let $N=\nuc(Q)$. If $N=1$ then Proposition \ref{Pr:Proper} yields a proper subloop of $Q$ whose order is divisible by $3$, and we are done by the induction assumption. Suppose from now on that $N\ne 1$. If $3$ divides $|N|$ then $N$ contains an element of order $3$ since $N$ is a group. Else $3$ divides $|Q/N|$, and by the induction assumption there is $x\in Q$ such that $xN$ is of order $3$ in $Q/N$. Since $xN$ is the homomorphic image of $x$ under the natural projection modulo $N$, the order of $xN$ divides the order of $x$. A suitable power of $x$ is then of order $3$.


\begin{thebibliography}{99}

\bibitem{AlbertII}
A.A.~Albert, \emph{Quasigroups. II.}, Trans. Amer. Math. Soc. \textbf{55} (1944), 401--419.

\bibitem{Barnes}
M.~Barnes, \emph{On loop commutators, quaternionic automorphic loops, and related topics}, PhD dissertation, Department of Mathematics, University of Denver, May 2022.

\bibitem{BruckTrans}
R.H.~Bruck, \emph{Contributions to the theory of loops}, Trans. Amer. Math. Soc. \textbf{60} (1946), 245--354.

\bibitem{BruckBook}
R.H.~Bruck, \emph{A survey of binary systems}, Gruppentheorie Ergebnisse der Mathematik und ihrer Grenzgebiete, (N.F.), Heft \textbf{20}, Springer-Verlag, Berlin-G\"ottingen-Heidelberg 1958.

\bibitem{CheinRobinson}
O.~Chein and D.A.~Robinson, \emph{An ``extra'' law for characterizing Moufang loops}, Proc. Amer. Math. Soc. \textbf{33} (1972), 29--32.

\bibitem{Csorgo}
P.~Cs\"org\H{o}, \emph{Every Moufang loop of odd order has nontrivial nucleus}, J.~Algebra \textbf{603} (2022), 89--117.

\bibitem{Doro}
S.~Doro, \emph{Simple Moufang loops}, Math. Proc. Cambridge Philos. Soc. \textbf{83} (1978), no. \textbf{3}, 377--392.

\bibitem{Drapal}
A.~Dr\'apal, \emph{On left conjugacy closed loops with a nucleus of index two}, Abh. Math. Sem. Univ. Hamburg \textbf{74} (2004), 205--221.

\bibitem{DrapalMT}
A.~Dr\'apal, \emph{A simplified proof of Moufang's theorem}, Proc. Amer. Math. Soc. \textbf{139} (2011), no. \textbf{1}, 93--98.

\bibitem{FeitThompson}
W.~Feit and J.G.~Thompson, \emph{Solvability of groups of odd order}, Pacific J.~Math. \textbf{13} (1963), 775--1029.

\bibitem{Fenyves}
F.~Fenyves, \emph{Extra loops. I.}, Publ. Math. Debrecen \textbf{15} (1968), 235-238.

\bibitem{FM}
R.~Freese and R.~McKenzie, \emph{Commutator theory for congruence modular varieties}, London Mathematical Society Lecture Note Series \textbf{125}, Cambridge University Press, Cambridge, 1987.

\bibitem{GagHalLagrange}
S.M.~Gagola, III and J.I.~Hall, \emph{Lagrange's theorem for Moufang loops}, Acta Sci. Math. (Szeged) \textbf{71} (2005), no. \textbf{1}--\textbf{2}, 45--64.

\bibitem{GAP}
The GAP~Group, \emph{GAP -- Groups, Algorithms, and Programming, Version 4.4.12}; 2008, \url{http://www.gap-system.org}.

\bibitem{GlaubermanII}
G.~Glauberman, \emph{On loops of odd order. II.}, J.~Algebra \textbf{8} (1968), 393--414.

\bibitem{GriZavLagrange}
A.N.~Grishkov and A.V.~Zavarnitsine, \emph{Lagrange's theorem for Moufang loops}, Math. Proc. Cambridge Philos. Soc. \textbf{139} (2005), no. \textbf{1}, 41--57.

\bibitem{GriZavSylow}
A.N.~Grishkov and A.V.~Zavarnitsine, \emph{Sylow's theorem for Moufang loops}, J. Algebra \textbf{321} (2009), no. \textbf{7}, 1813--1825.

\bibitem{Hall}
J.I.~Hall, \emph{Moufang Loops and Groups with Triality are Essentially the Same Thing}, Mem. Amer. Math. Soc. \textbf{260} (2019), no. 1252.

\bibitem{Moufang}
R.~Moufang, \emph{Zur Struktur von Alternativk\"orpern}, Math. Ann. \textbf{110} (1935), no. \textbf{1}, 416--430.

\bibitem{Liebeck}
M.W.~Liebeck, \emph{The classification of finite simple Moufang loops},  Math. Proc. Cambridge Philos. Soc. \textbf{102} (July 1987), issue \textbf{1} , 33--47.

\bibitem{LOOPS}
G.P. Nagy and P. Vojt\v{e}chovsk\'y, LOOPS, version 3.4.1, package for GAP, \url{https://github.com/gap-packages/loops}

\bibitem{PflugfelderBook}
H.O.~Pflugfelder, \emph{Quasigroups and loops: introduction}, Sigma Series in Pure Mathematics \textbf{7}, Heldermann Verlag, Berlin, 1990.

\bibitem{Paige}
L.J.~Paige, \emph{A class of simple Moufang loops}, Proc. Amer. Math. Soc. \textbf{7} (1956), 471--482.

\bibitem{Rotman}
J.J.~Rotman, \emph{An introduction to the theory of groups}, fourth edition, Graduate Texts in Mathematics, Springer-Verlag, $1995$.

\bibitem{StaVojComm}
D.~Stanovsk\'y and P.~Vojt\v{e}chovsk\'y, \emph{Commutator theory for loops}, J.~Algebra \textbf{399} (2014), 290--322.

\bibitem{StaVojAbel}
D.~Stanovsk\'y and P.~Vojt\v{e}chovsk\'y, \emph{Abelian extensions and solvable loops}, Results Math. \textbf{66} (2014), 367--384.

\end{thebibliography}
\end{document}